\documentclass[12pt]{article}

\usepackage{amsmath,amssymb,latexsym,theorem,enumerate}
\usepackage[centering,dvips]{geometry}
\usepackage[ps2pdf, bookmarks=true,
            bookmarksnumbered=true, breaklinks=true,
            pdfstartview=FitH, hyperfigures=false,
            plainpages=false, naturalnames=true,
            colorlinks=true,pagebackref=true,
            pdfpagelabels]{hyperref}
\usepackage[usenames]{color}
\def\xybiglabels{\def\labelstyle{\textstyle}}

{\theorembodyfont{\rmfamily}\newtheorem{example}{Example}[section]}
{\theorembodyfont{\rmfamily}\newtheorem{rem}[example]{Remark}}
{\theorembodyfont{\rmfamily}}
{\theorembodyfont{\rmfamily}}
{\theorembodyfont{\rmfamily}}
\newtheorem{prop}[example]{Proposition}

\newtheorem{thm}[example]{Theorem}

\usepackage[all]{xy}

\def\xybiglabels{\def\labelstyle{\textstyle}}

\def\Cok{\operatorname{Cok}}

\def\Ker{\operatorname{Ker}}

\def\Z{\mathbb{Z}}

\def\geq{\geqslant}

\newenvironment{proof}{\noindent {\bf Proof} }{\hfill $\Box$
}

\def\Crs{\mathsf{Crs}}
\def\CRS{\mathsf{CRS}}

\def\K{\mathbb{K}}

\newcommand{\labto}[1]{\xrightarrow{\labelstyle\textstyle #1}}

\newcommand{\xdirects}[2]{\def\objectstyle{\scriptstyle} \objectmargin={0pt}
\xy
(0,0)*+{}="a",(0,-6)*+{\rule{0em}{1.5ex}#2}="b",(7,0)*+{\;#1}="c"
\ar@{->} "a";"b" \ar @{->}"a";"c" \endxy }

\def\M{\mathcal M}
\def\cF{\mathcal F}
\def\d{\delta}
\begin{document}

\title{Crossed modules \\ and  the homotopy 2-type of a free loop space}
\author{Ronald Brown\thanks{School of Computer Science, University of Bangor, LL57 1UT,  Wales}} \maketitle
\begin{center}
  Bangor University  Maths Preprint 10.01
\end{center}

\begin{abstract}
The question was asked by Niranjan Ramachandran: how to describe the
fundamental groupoid of $LX$, the free loop space of a space $X$? We
show how this depends on the homotopy 2-type of $X$  by assuming $X$
to be the classifying space of a crossed module over a group, and
then describe completely a crossed module over a groupoid
determining the homotopy 2-type of $LX$;  that is we describe
crossed modules representing the 2-type of each component of $LX$.
The method requires detailed information on the monoidal closed
structure on the category of crossed
complexes.\footnote{MSClass:18D15,55Q05,55Q52; Keywords: free loop
space, crossed module, crossed complex, closed category, classifying
space, higher homotopies.}
\end{abstract}

\section{Introduction}
It is well known that for a connected $CW$-complex $X$ with
fundamental group $G$  the set of components of the free loop space
$LX$ of $X$ is bijective with the set of conjugacy classes of the
group $G$, and that the fundamental groups of $LX$ fit into a family
of exact sequences derived from the fibration $LX \to X$ obtained by
evaluation at the base point.

Our aim is to describe the homotopy 2-type of $LX$, the free loop
space on $X$,  when $X$ is a connected $CW$-complex, in terms of the
2-type of $X$. Weak homotopy 2-types are described by crossed
modules (over groupoids), defined in \cite{BH81:algcub} as follows.

A {\it crossed module} $\M$ is a morphism $\delta: M \to P$ of
groupoids which is the identity on objects such that $M$ is just a
disjoint union of groups $M(x), x \in P_0$, together with an action
of $P$ on $M$ written $(m,p) \mapsto m^p$, $m \in M(x), p:x \to y$
with $m^p \in M(y)$ satisfying the usual rules for an action. We
find it convenient to use (non-commutative) additive notation for
composition so if $p:x \to y, q:y \to z$ then $p+q: x \to z$, and
$(m+n)^p= m^p+n^p, (m^p)^q=m^{p+q}, m^0=m$. Further we have the two
crossed module rules for all $p \in P, m,n \in M$:
\begin{enumerate}[CM1)]
\item $\delta(m^p)= -p+ \delta m +p$;
\item $-n+m+n=m^{\delta n}$;
\end{enumerate}
whenever defined. This is a {\it crossed module of groups} if $P_0$
is a singleton.

A crossed module $\M$  as above has a simplicial nerve $K=N^\Delta
\M$ which in low dimensions is described as follows:
\begin{itemize}
  \item $K_0=P_0$;
  \item $K_1=P$;
  \item  $K_2$ consists of quadruples $\sigma=(m;c,a,b)$ where $m
 \in M, a,b,c \in P$ and $\delta m= -c+a+b$ is well defined;
 \item $K_3$ consists of quadruples
 $(\sigma_0,\sigma_1,\sigma_2,  \sigma_3)$ where $\sigma_i \in K_2$
 and the $\sigma_i$ make up the faces of a 3-simplex, as shown in the following diagrams:
\end{itemize}
$$\xybiglabels
 \vcenter
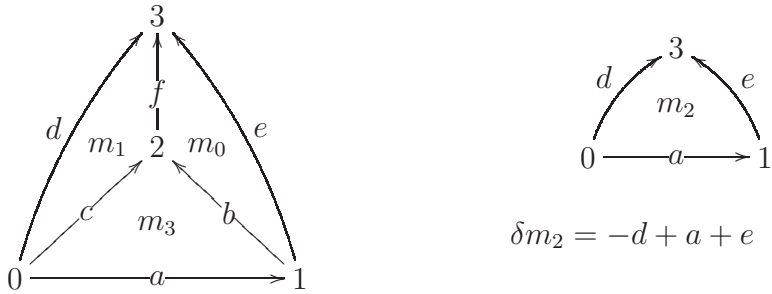
{\xymatrix@R=3pc@C=3.5pc {& 3& \\
& 2 \ar[u] |(0.4)f \ar@{}[r]|(0.35){m_0}\ar@{}[l]|(0.35){m_1}& \\
0\ar [rr] |a \ar [ur]|{c} \ar@/^0.6pc/ [uur]^{d}
&\ar@{}[u]|(0.4){m_3}& 1 \ar @/_0.6pc/[uul]_{e} \ar [ul] |b
}}\hspace{6em}
\vcenter{\xymatrix@C=1.8pc@R=2.3pc{&\ar @{}[d]|(0.55){m_2}3& \\
0 \ar@/^0.5pc/ [ur] ^{d} \ar [rr] |a & & 1 \ar @/_0.5pc/ [lu] _e\\
\ar@{}[r]^{\delta m_2=-d+a+e}&}}$$ providing we have the rules
\begin{alignat*}{2}
\mu m_0&= -e+b+f,\quad  & \mu m_1&= -d+c+f,\\
\mu m_2&= -d+a+e, & \mu m_3&= -c+a+b,
\end{alignat*}
together with the rule
\begin{equation*}
(m_3)^f-m_0 -m_2+m_1=0.
\end{equation*}
You may like to verify that these rules are consistent.

A crossed module is the dimension 2 case of a {\it crossed complex},
the definition of which in the single vertex case goes back to
Blakers in \cite{BL48}, there called a `group system',   and in the
many vertex case is in \cite{BH81:algcub}. The definition of the
nerve of a crossed complex $C$ in the one vertex case is also in
\cite{BL48}, and in the general case is in \cite{As78,BH91}. An
alternative description of $K$ is that $K_n$ consists of the crossed
complex morphisms $\Pi \Delta^n_* \to \M$ where $\Pi \Delta^n_* $ is
the fundamental crossed complex of the $n$-simplex, with its
skeletal filtration, and $\M$ is also considered as a crossed
complex trivial in dimensions $>2$. This shows the analogy with the
Dold-Kan theorem for chain complexes and simplicial abelian groups,
\cite{Do58}.

We thus  define the {\it classifying space  $B \M$ of $\M$}  to be
the geometric realisation $|N^\Delta \M|$, a special case of the
definition in \cite{BH91}. It follows that an $a \in P(x)$ for some
$x \in P_0$ determines a $1$-simplex in $X=B \M$ which is a loop and
so a map $a': S^1 \to B \M$, i.e. $a' \in LX$.

The chief properties of $X=B\M$ are that $\pi_0(X) \cong \pi_0(P)$
and  for each $x \in P_0$
$$\pi_i(X,x) \cong \begin{cases}
  \Cok(\delta:M(x) \to P(x))  & \text{ if } i=1, \\
  \Ker(\delta :M(x) \to P(x)) & \text{ if } i=2, \\
  0 & \text{ if } i >2.
\end{cases}$$
Further if $Y$ is a  $CW$-complex, then there is a crossed module
$\M$ and a map $ Y \to B \M$ inducing isomorphisms of $\pi_0, \pi_1,
\pi_2$. For an exposition of some basic facts on crossed modules and
crossed complexes in relation to homotopy theory, see for example
\cite{Brown-grenoble}. There are other versions of the classifying
space, for example the cubical version given in \cite{BHS}, and one
for crossed module of groups using the equivalence of these with
groupoid objects in groups, see for example
\cite{Lod82,baez-stev-class-2-groups}. However the latter have not
been shown to lead  to the homotopy classification Theorem
\ref{thm:homclass} below.

Our main result is:
\begin{thm}\label{thm:combined}
Let  $\mathcal M$ be the crossed module of groups  $\delta: M \to P$
and let  $X=B\mathcal M$ be the classifying space of $\M$. Then the
components of $LX$, the free loop space on $X$,  are determined by
equivalence classes of elements $a \in P$ where $a,b$ are equivalent
if and only if there are elements $m \in M, p \in P $ such that
$$b= p + a + \delta m -p. $$ Further the homotopy $2$-type of a
component of $LX$ given by $a \in P$ is determined by the crossed
module of groups  $L\M [a]=(\delta_a: M \to P(a))$ where
\begin{enumerate}[\rm (i)]
\item $P(a)$ is the group of   elements $(m,p)\in M \times P$
such that $\delta m= [a,p]$, with composition $(n,q)+(m,p)=
(m+n^p,q+p)$;
\item $\delta_a(m)= ( -m^a + m,\delta m)$, for $m \in M$;
\item the action of $P(a)$ on $M$ is given by $n^{(m,p)}= n^p$ for $n \in M, (m,p) \in P(a)$.
\end{enumerate}
In particular $\pi_1(LX,a)$ is isomorphic to $\Cok \delta_a$, and
$\pi_2(LX,a) \cong \pi_2(X,*)^{\bar{a}}$, the elements of
$\pi_2(X,*)$ fixed under the action of $\bar{a}$, the class of $a$
in $G=\pi_1(X,*)$.
\end{thm}
We give a detailed proof that $L \M[a]$ is a crossed module in
Appendix \ref{app:proof}.

\begin{rem}
The composition in (i) can be seen geometrically in the following
diagram:
\begin{equation}
\xybiglabels \vcenter{\xymatrix@M=0pt@=3pc{\ar [d] _q \ar [r]^a \ar@{} [dr]|n & \ar [d]^q \\
\ar [r] |a \ar [d]_p \ar@{}[dr] |m & \ar [d]^p \\
\ar [r]_a& }}\quad  = \quad
\vcenter{\xymatrix@M=0pt@C=4pc@R=3pc{\ar[d]_{q+p}\ar[r] ^a
\ar@{}[dr]|{m+n^p}& \ar [d]^{q+p} \\
\ar [r]_a& }} \qquad \xdirects{2}{1}\tag*{$\Box$}\end{equation}
\end{rem}
The following examples are due to C.D. Wensley.
\begin{example}
 $\delta=0: M \to P$, so that $M$ is a $P$-module. Then $P(a)$ is
the set of $(m,p)$ s.t. $[a,p]=0$, i.e. $p \in C_a(P)$, and so is $M
\rtimes C_a(P)$.  ($P=G$ the fundamental group, as $\delta=0$). But
$\delta_a(m)=(-m^a+m,0)$. So $\pi_1(L\M,a) =(M/[a,M]) \rtimes
C_a(P)$. \hfill $\Box$
\end{example}
\begin{example}
If $a \in Z(P)$, the center of $P$, then $[a,p]=0$ for all $p$. (For
example, $P$ might be abelian.) Hence $P(a)= \pi \rtimes P$. Then
$\pi_1(L\M,a) = (\pi \rtimes P)/ \{(-m^a+m,\delta m)\mid m \in M\}.$

It is not clear to me that even in this case the exact sequence
splits. (??)\hfill $\Box$
\end{example}

It is also possible to give a less explicit description of
$\pi_1(LX,a)$ as part of an exact sequence:
\begin{thm}\label{thm:exact}
 Under the circumstances of Theorem \ref{thm:combined},
if we set $\pi = \Ker \delta = \pi_2(X), G= \Cok \delta= \pi_1(X)$,
with the standard module action of $G$ on $\pi$, then the
fundamental group $\pi_1(LX,a)$ in the component given by $a \in P$
is part of an exact sequence:
\begin{equation} 0 \to \pi^{\bar{a}} \to \pi \to \pi/\{\bar{a}\} \to \pi_1(LX, a) \to C_{\bar{a}}(G) \to 1
\end{equation}
where:  $\pi/\{\alpha\}$ denotes $\pi$ with the action of $\alpha$
killed;  and $C_\alpha(G)$ denotes the centraliser of the element
$\alpha \in G$.
\end{thm}

The proof of Theorem \ref{thm:combined}, which  will be given in
Section 2, is essentially an exercise in the use of the following
classification theorem \cite[Theorem A]{BH91}:
\begin{thm}\label{thm:homclass}
Let $Y$ be a $CW$-complex with its skeletal filtration $Y_*$ and let
$C$ be a crossed complex, with its classifying space written $BC$.
Then there is a natural weak homotopy equivalence
$$ B(\CRS(\Pi Y_*,C)) \to (BC)^Y. $$
  \end{thm}
In the statement of this theorem we use the internal hom $\CRS(-,-)$
in the category $\Crs$ of crossed complexes: this  internal hom is
described explicitly in \cite{BH87}, in order to set up the
exponential law
$$\Crs(A \otimes B,C) \cong \Crs(A,\CRS(B,C))$$ for crossed
complexes $A,B,C$, i.e. to give a monoidal closed structure on the
category $\Crs$. Note that $\CRS(B,C)_0= \Crs(B,C)$, $\CRS(B,C)_1$
gives the homotopies of morphisms, and $\CRS(B,C)_n$ for $n \geq 2$
gives the higher homotopies.

\section{Proofs}

We deduce Theorem \ref{thm:combined} from the following Theorem.

\begin{thm}\label{thm:mainthm}
Let $X=B\mathcal M$, where $\mathcal M$ is the crossed module of
groups  $\delta: M \to P$. Then  the homotopy $2$-type of $LX$, the
free loop space of $X$, is described by the crossed module over
groupoids $L \mathcal M$  where \begin{enumerate}[\rm (i)]
  \item $(L \mathcal M)_0 = P$;
  \item $(L \mathcal M)_1= M \times
  P \times P$ with source and target given by
$$s(m,p,a)= p+a+\delta m -p, \quad  t(m,p,a)= a$$
for $a,p\in P, m \in M$;
\item the composition of such triples is given by
$$ (n,q,b)+(m,p,a) =( m+n^p,q+p,a)$$
which of course is defined under the condition that \begin{align*}
b&=p+a+\delta m -p
\end{align*}
or, equivalently, $ b^p= a +\delta m $;
\item if $a \in P$ then $(L\mathcal M)_2(a)$ consists of pairs
$(m,a)$ for all $m \in M$, with addition and boundary
$$(m,a)+(n,a)=(m+n,a), \qquad \delta (m,a)=(-m^a +m,\delta m, a);$$
\item the action of $(L\mathcal M)_1$ on $(L\mathcal M)_2$ is given
by: $(n,b)^{(m,p,a)}$ is defined if and only if $b^p=a+\delta m $
and then its value is $(n^p,a)$.
  \end{enumerate}
\end{thm}
\begin{proof}
In Theorem \ref{thm:homclass} we set $Y=S^1$ with its standard cell
structure $e^0 \cup e^1$, and can write $\Pi Y_* \cong \K(\Z,1)$
where the latter is the crossed complex with a base point $z_0$ and
a free generator $z$ in dimension $1$, and otherwise trivial. Thus
morphisms of crossed complexes from $\K(\Z,1)$,  and homotopies and
higher homotopies of such morphisms, are completely determined by
their values on $z_0$ and on $z$.

A crossed module over a group or groupoid is also regarded as a
crossed complex trivial in dimensions $> 2$.

All the formulae required to prove Theorem \ref{thm:mainthm} follow
from those  for the internal hom $\CRS$ on the category $\Crs$ given
in \cite[Proposition 3.14]{BH87} or \cite[\S 7.1.vii, \S 9.3]{BHS}.

We set $L\M= \CRS(\K(\Z,1),\M)$.

Since $\K(\Z,1)$ is a free crossed complex with one generator $z$ in
dimension 1, the elements $a \in P$ are bijective with the morphisms
$f:\K(\Z,1) \to \mathcal M$, and we write this bijection as $a
\mapsto \hat{a}$, where $a=\hat{a}(z)$. Also the homotopies and
higher homotopies from $\K(\Z,1) \to \mathcal M$ are determined by
their values on $z$ and on the element $z_0$ of $\K(\Z,1)$ in
dimension 0. Thus a 1-homotopy $(h,\hat{a}):\hat{b}\simeq \hat{a}$
is such that $h$ lifts dimension by 1, and  is given by elements
$p=h(z_0)\in P, m=h(z) \in M$ and so $(h,\hat{a})$  is given by a
triple $(p,m,a)$. The condition that this triple  gives a homotopy
$\hat{b} \simeq \hat{a}$ translates to
$$ b=p+a+\delta m-p$$
or, equivalently, $ a+\delta m= b^p$. It follows  easily that
$\hat{b},\hat{a}$ belong to the same component of $L\M$ if and only
if $b,a$ give conjugate elements in the quotient group
$\pi_1({\mathcal M})$. (The use of such  general homotopies was
initiated in \cite{W49:CHII}.)

The composition of such homotopies $\hat{c} \simeq \hat{b}\simeq
\hat{a}$ is given by:
$$(n,q,b) + (m,p,a)= (m+n^p,q+p,a)$$
which of course is defined if and only if
$$b^p=a+\delta m.$$

A 2-homotopy $(H,\hat{a})$ of $\hat{a}$ is such that $H$ lifts
dimension by 2 and so is given by an element $H(z_0) \in M$. There
are rules giving the composition, actions, and boundaries of such 1-
and 2-homotopies.

In particular  the action of a 1-homotopy $(h,f^+):f^- \simeq f^+$
on a 2-homotopy $(H,f^-)$ gives a 2-homotopy $(H^h,f^+)$ where
$H^h(c)=H(c)^{h(tc)}$. Here we take $c=z_0$ so that we obtain the
action $(n,b)^{(m,p,a)}=n^p$.

All these formulae follow from those given in  \cite[Proposition
3.14]{BH87} or \cite[\S 9.3]{BHS}.

A 2-homotopy $(H,\hat{a})$ is given by $a=\hat{a}(z)$ and
$m=H(z_0)\in M$. We then have to work out $\delta_2(H)$. We find
that
\begin{align*}
 \delta_2(H)(x)&= \begin{cases}
  \delta H(z_0) & \text{ if } x=z_0,\\
  -H(sz)^{\hat{a}(z)} + H(tz) + \delta H(z) & \text{ if } x=z,
\end{cases}\\
& = \begin{cases}  \delta  m & \text{ if } x=z_0,\\
                   -m^a +m & \text{ if }  x=z.
\end{cases} \end{align*}

This completes the proof of Theorem \ref{thm:mainthm}.
\end{proof}
\medskip

The proof of  Theorem \ref{thm:combined} now follows by restricting
the crossed module of groupoids given in Theorem \ref{thm:mainthm}
to  $L\M(a)$, the crossed module of groups over the object $a \in
(L\M)_1=P$. Then we have an isomorphism $\theta: L\M(a) \to L\M[a]$
given by $\theta_0(a)=*, \; \theta_1(m,p,a)=(m,p), \;
\theta_2(m,a)=m$.

For the next result we need the notion of fibration of crossed
modules of groupoids which is a special case of fibrations of
crossed complexes as defined in \cite{How79} and applied in
\cite{brown-homclass}.

\begin{thm}\label{thm:fibration}
In the situation of Theorem \ref{thm:mainthm}, there is a fibration
$L\M \to \M$ of crossed modules of groupoids.  Hence if $$\pi=
\pi_2(X)\cong \Ker \delta,\quad  G = \pi_1(X)\cong \Cok \delta$$
then for each $a \in P$ there is an exact sequence
\begin{equation} 0 \to \pi^{\bar{a}} \to \pi \to \pi/\{\bar{a}\} \to \pi_1(LX, a') \to C_{\bar{a}}(G) \to 1
\end{equation}
where: $\bar{a}$ denotes the image of $a'$ in $G$; $\pi/\{\alpha\}$
denotes $\pi$ with the action of $\alpha$ killed;  and $C_\alpha(G)$
denotes the centraliser of the element $\alpha \in G$.
\end{thm}
\begin{proof}
We  define the fibration $\psi: L \M \to \M$ by the inclusion
$i:\{z_0\} \to \K(\Z,1)$ and the identification $\CRS(\{z_0\}, \M)
\cong \M$, where here $\{z_0\}$ denotes also the trivial crossed
complex on the point $z_0$. Then $\psi$ is a fibration since  $i$ is
a cofibration, see \cite{BG89}. The exact description of $\psi$ in
terms given earlier is that
\begin{align*}
\psi_0 (a)&= *, \quad a \in P,\\
\psi_1(m,p,a)&= p, \quad (m,p,a) \in M \times P\times P, \\
\psi_2(n,a) &= n, \quad (n,a) \in M \times P.
\end{align*}
To say that $\psi$ is a fibration of crossed modules over groupoids
is to say that: (i) it is a morphism; (ii) $(\psi_1, \psi_0)$ is a
fibration of groupoids, \cite{B70,anderson-fibrations};  and (iii)
$\psi_2$ is piecewise surjective.

Let $\cF$ denote the fibre of $\psi$. Then $$\cF_0=P, \quad
\cF_1=\{0\}\times M \times P,\quad  \cF_2= \{0\}\times P.$$ The
exact sequence of the fibration for a given base point $a \in
\cF_0=P$ is
\begin{multline*}
  0 \to \pi_2(\cF,a) \to \pi_2(L \M, a) \to \pi_2(\M,*) \labto{\, \partial \,} \\
  \to \pi_1(\cF,a) \to \pi_1(L \M, a) \to \pi_1(\M,*)
  \labto{\, \partial\, } \pi_0(\cF) \to \pi_0 (L\M) \to *.
\end{multline*}
Under the obvious identifications, this leads to the exact sequence
of Theorem \ref{thm:exact}.
\end{proof}
\begin{rem}
These results and methods should be related to the description in
\cite[\S 6]{B-87} of the homotopy type of the function space
$(BG)^Y$ where  $G$ is an abstract group and $Y$ is a $CW$-complex,
and which gives a result of Gottlieb in \cite{Got}. \hfill $\Box$
\end{rem}

\begin{rem}
Here is a methodological point. The category $\Crs$ of crossed
complexes is equivalent to that of $\infty$-groupoids, as in
\cite{BH81:inf}, where these $\infty$-groupoids  are now commonly
called `strict globular $\omega$-groupoids'. However the internal
hom in the latter category is bound to be more complicated than that
for crossed complexes, because the cell structure of the standard
$n$-globe, $n >1,$
$$E^n = e^0_\pm \cup e^1_\pm \cup \cdots \cup e^{n-1}_\pm \cup e^n$$
is more complicated than that for the standard cell for which
$$E^n=e^0 \cup e^{n-1}\cup e^n, n>1.$$
Also we obtain a precise answer using filtered spaces and strict
structures, whereas the current fashion is to go for weak structures
as yielding more homotopy $n$-types for $n>2$. In fact many  results
on crossed complexes are obtained using cubical methods.   \hfill
$\Box$
\end{rem}
\section*{Appendix: Verification of crossed module rules}
\label{app:proof} We now verify  the crossed module rules for the
structure $$L\M[a]= (M \labto{\d _a} P(a))$$ defined in Theorem
\ref{thm:combined} from a crossed module of groups $\M= (M \labto{\d
} P)$ and $a\in P$ as follows:
\begin{align*}P(a)&=\{(m,p)\in M \times P\mid \delta m= -[a,p]=-a-p+a+p\};\\
\delta _a m& = (-m^a+m, \delta m);\\
(n,q)+(m,p)&=(m+n^p,q+p);\\
n^{(m,\,p)}&=n^p.
\end{align*}

\begin{prop}If $\delta:M \to P$ is a crossed module of groups, and $a \in P$, then
$L \M[a]$ as defined above is also a crossed module of groups.
\end{prop}
\noindent {\bf Proof} It is easy to check that $\delta (-m^a+m)=
[a,\d m]$, so that $\delta_a(m) \in P(a)$.

 We next show that $\d _a$ is a morphism:
\begin{align*}
  \delta_a(n) + \delta_a(m)&=(-n^a+n,\delta n)+ (-m^a+m,\delta m)\\
  &= (-m^a +m +(-n^a+n)^{\delta m}, \delta n + \delta m) \\
  &= (-m^a -n^a +n +m, \d n + \d m) \\
  &= \d _a (n+m).\\
\intertext{Now we verify the first crossed module rule. Let $(m,p)
\in P(a), n \in M$:}
  -(m,p) + \d _a n + (m,p)&= (-m^{-p},-p) +(-n^a+n,\d n)+ (m,p) \\
  &= (-n^a +n +(-m^{-p})^{\d n}, -p +\d n)+(m,p) \\
  &= (-n^a -m^{-p} + n,-p + \d n)+ (m,p) \\
  &= (m+(-n^a-m^{-p}+n)^p,-p +\d n +p)\\
  &= (m-n^{a+p}-m+n^p,\d (n^p))\\
  &= (-n^{a+p-\d m} +n^p ,\d (n^p))\\
  &= (-n^{p+a} +n^p,\d (n^p))\tag*{since $\delta m= [a,p]$}\\
  &= \d_a (n^p).\\
\intertext{Now we verify the second crossed module rule:}
  m^{\d _a n}&= m^{(-n^a+n,\,\delta n)}\\
  &= m^{\d n}\\
  &= -n+m+n. \tag*{$\Box$}
\end{align*}

In effect, this illustrates  that verifying the crossed complex
rules for the internal hom $\CRS(C,D)$  is possible but tedious, and
that is it is easier to say it follows from the general construction
in terms of $\omega$-groupoids and the equivalence of categories, as
in \cite{BH87}. On the other hand, this direct proof  `proves', in
the old sense of `tests', the general theory.

\end{document}